\documentclass[]{amsart}

\usepackage[english]{babel}
\usepackage[utf8]{inputenc}
\usepackage[T1]{fontenc}
\usepackage{paralist}
\usepackage{amsmath, amsfonts, amssymb, amsthm}
\usepackage{color}
\usepackage[normalem]{ulem}
\usepackage{tikz}

\newcommand{\NN}{\mathbb{N}}

\newcommand{\QQ}{\mathbb{Q}}

\newcommand{\ZZ}{\mathbb{Z}}

\newcommand{\E}{\mathcal{E}}

\newcommand{\Hc}{\mathcal{H}}

\newcommand{\set}[1]{\left\{ #1 \right\}}
\newcommand{\setb}[1]{\left( #1 \right)}
\newcommand{\abs}[1]{\left| #1 \right|}

\newcommand{\bino}[2]{\begin{pmatrix} #1 \\ #2 \end{pmatrix}}

\newtheorem{mymasterthm}{notForUse}
\theoremstyle{definition}

\newtheorem{myrem}[mymasterthm]{Remark}

\theoremstyle{plain}
\newtheorem{mylemma}[mymasterthm]{Lemma}
\newtheorem{mythm}[mymasterthm]{Theorem}

\allowdisplaybreaks

\title[Approximation and power sums]{Approximation of values of algebraic elements over the ring of power sums}
\makeatletter
\@namedef{subjclassname@2020}{\textup{2020} Mathematics Subject Classification}
\makeatother

\author[C. Fuchs]{Clemens Fuchs}
\address{Clemens Fuchs\newline
	\indent University of Salzburg\newline
	\indent Department of Mathematics\newline
	\indent Hellbrunnerstr. 34 \newline
	\indent A-5020 Salzburg, Austria}
\email{clemens.fuchs@plus.ac.at}

\author[S. Heintze]{Sebastian Heintze}
\address{Sebastian Heintze\newline
	\indent Graz University of Technology\newline
	\indent Institute of Analysis and Number Theory\newline
	\indent Steyrergasse 30/II \newline
	\indent A-8010 Graz, Austria}
\email{heintze@math.tugraz.at}

\thanks{Supported by Austrian Science Fund (FWF): I4406.}

\begin{document}
	
	\maketitle
	
	
	\begin{abstract}
		Let $ \QQ\E_{\ZZ} $ be the set of power sums whose characteristic roots belong to $ \ZZ $ and whose coefficients belong to $ \QQ $, i.e. $ G : \NN \rightarrow \QQ $ satisfies
		\begin{equation*}
			G(n) = G_n = b_1 c_1^n + \cdots + b_h c_h^n
		\end{equation*}
		with $ c_1,\ldots,c_h \in \ZZ $ and $ b_1,\ldots,b_h \in \QQ $.
		Furthermore, let $ f \in \QQ[x,y] $ be absolutely irreducible and $ \alpha : \NN \rightarrow \overline{\QQ} $ be a solution $ y $ of $ f(G_n,y) = 0 $, i.e. $ f(G_n,\alpha(n)) = 0 $ identically in $ n $.
		Then we will prove under suitable assumptions a lower bound, valid for all but finitely many positive integers $ n $, for the approximation error if $ \alpha(n) $ is approximated by rational numbers with bounded denominator.
		After that we will also consider the case that $ \alpha $ is a solution of
		\begin{equation*}
			f(G_n^{(0)}, \ldots, G_n^{(d)},y) = 0,
		\end{equation*}
		i.e. defined by using more than one power sum and a polynomial $ f $ satisfying some suitable conditions.
		This extends results of Bugeaud, Corvaja, Luca, Scremin and Zannier.
	\end{abstract}
	
	\section{Introduction}
	
	The theory of Diophantine approximation has a long history.
	After results of Liouville and Thue, a famous milestone result was Roth's theorem.
	It states that for any algebraic number $ \alpha $ and arbitrary $ \varepsilon > 0 $ there is a positive constant $ c(\alpha,\varepsilon) $, depending on $ \varepsilon $ and the approximated number $ \alpha $, such that for all rational integers $ p,q $ with $ q > 0 $ and $ p/q \neq \alpha $ the lower bound
	\begin{equation*}
		\abs{\alpha - \frac{p}{q}} \geq \frac{c(\alpha,\varepsilon)}{q^{2+\varepsilon}}
	\end{equation*}
	holds.
	The appearing constant $ c(\alpha,\varepsilon) $ is ineffective, i.e. it is not known how the constant could be computed.
	For a more detailed introduction to Diophantine approximation we refer to \cite{zannier-2009}.
	
	In Roth's theorem there is only a single algebraic number $ \alpha $ which has to be approximated by rational numbers.
	It is a natural question to ask what can be said about the approximation of several numbers which are parametrized in some way, say by a positive integer $ n $.
	For references to other such results ``with moving targets'' we refer to the remarks (and papers cited) in \cite{corvaja-zannier-2005}.
	Obviously, we could apply Roth's result to each number separately, but in this case the dependence of the constant $ c(\alpha(n),\varepsilon) $ on the parameter $ n $ cannot be controlled.
	Therefore we are interested in (another) approximation result which gives a precisely described dependence on the parameter $ n $.
	
	Corvaja and Zannier proved in \cite{corvaja-zannier-2005} that for power sums $ G^{(1)} $ and $ G^{(2)} $ a positive constant $ k $ exists such that under suitable assumptions for all but finitely many positive integers $ n $ and for integers $ p,q $ with $ q $ positive and not too large we have
	\begin{equation*}
		\abs{\frac{G^{(1)}(n)}{G^{(2)}(n)} - \frac{p}{q}} \geq \frac{1}{q^k} e^{-n\varepsilon}.
	\end{equation*}
	Then they used this approximation result to deduce that under the given conditions the length of the continued fraction for $ G^{(1)}(n) / G^{(2)}(n) $ tends to infinity.
	
	A similar approximation statement for the square root of a power sum was given by Bugeaud and Luca in \cite{bugeaud-luca-2005} in order to prove, again under suitable assumptions, that the length of the period of the continued fraction expansion of $ \sqrt{G(n)} $ tends to infinity.
	Finally, Scremin \cite{scremin-2006} combined these two approximation results and showed that under assumptions which are very similar to those of Theorem \ref{thm:approxresult} below (hence we will not repeat them here in the introduction as they are a little bit technical) for three power sums $ G^{(1)}, G^{(2)}, G^{(3)} $ the lower bound
	\begin{equation*}
		\abs{\frac{\sqrt{G^{(1)}(n)} + G^{(2)}(n)}{G^{(3)}(n)} - \frac{p}{q}} \geq \frac{1}{q^k} e^{-n\varepsilon}
	\end{equation*}
	holds with finitely many exceptions.
	
	In the present paper we will generalize these approximation results to solutions $ \alpha(n) $ of $ f(G(n),y)=0 $ for an absolutely irreducible polynomial $ f \in \QQ[x,y] $ and a power sum $ G $.
	More precisely, we fix an algebraic closure $ \overline{\QQ} $ of $ \QQ $ and take, for every $ n $, an $ \alpha(n) $ which satisfies $ f(G(n),\alpha(n)) = 0 $. We say that $ \alpha(n) $ is a solution of $ f(G(n),y) = 0 $ meaning that $ \alpha : \NN \rightarrow \overline{\QQ} $ satisfies $ f(G(n),\alpha(n)) = 0 $ identically for every $ n $.
	Thus for $ f(x,y) = x-y^2 $ we get again the case $ \sqrt{G(n)} $ already covered by Bugeaud and Luca as well as Scremin.
	Since the continued fraction expansion has only for rational and quadratic algebraic numbers a special structure, i.e. it is finite and periodic, respectively, the application performed in the above mentioned papers cannot be generalized to our situation.
	Furthermore, we will consider the case that $ \alpha(n) $ is defined by
	\begin{equation*}
		f(G^{(0)}(n), \ldots, G^{(d)}(n),y) = 0
	\end{equation*}
	using more than one power sum and a polynomial $ f $ satisfying some suitable properties, and prove an approximation result in the same direction.
	Let us mention that this kind of generalization was already suggested in a final remark in \cite{corvaja-zannier-2005} by Corvaja and Zannier.
	
	\section{Notation and result}
	
	For rings $ B $ and $ C $ we denote by $ B\E_C $ the ring of power sums whose characteristic roots belong to $ C $ and whose coefficients belong to $ B $.
	In other words, we consider functions on the set of positive integers $ \NN $ of the form
	\begin{equation}
		\label{eq:sumGn}
		G_n := G(n) = b_1 c_1^n + \cdots + b_h c_h^n
	\end{equation}
	with $ c_1,\ldots,c_h \in C $ and $ b_1,\ldots,b_h \in B $.
	We write $ G_n \in B\E_C $ by slight abuse of language.
	In the special case of $ G_n \in \QQ\E_{\ZZ} $ we can, by splitting $ \NN $ into even and odd numbers, assume that all roots are positive.
	Hence we often suppose that $ c_1 > c_2 > \cdots > c_h > 0 $.
	
	Our first approximation result is now the following theorem which states that the values of the considered function $ \alpha $, which is algebraic over the ring $ \QQ\E_{\ZZ} $, along suitable arithmetic progressions either have good approximations by the values of a fixed element of $ \QQ\E_{\ZZ} $ for infinitely many $ n $, or cannot be approximated very good by (arbitrary) rational numbers with bounded denominator infinitely many times.
	
	\begin{mythm}
		\label{thm:approxresult}
		Let $ f(x,y) \in \QQ[x,y] $ be absolutely irreducible and $ G_n \in \QQ\E_{\ZZ} $ as in \eqref{eq:sumGn} with $ c_1 > c_2 > \cdots > c_h > 0 $ and $ c_1 > 1 $.
		Moreover, let $ \alpha(n) $ be a solution $ y $ of $ f(G_n,y) = 0 $.
		Then there exists a positive integer $ s $ such that for each fixed $ r \in \set{0,1,\ldots,s-1} $ we get the subsequent approximation result:
		There exists an integer $ k \geq 2 $ such that for any $ \varepsilon > 0 $ which is small enough, i.e.
		\begin{equation*}
			\varepsilon < \min \setb{\frac{1}{2(s+2)}, \frac{1}{2k}},
		\end{equation*}
		we have the following:
		
		If there is no power sum $ \eta \in \QQ\E_{\ZZ} $ with non-negative characteristic roots such that
		\begin{equation*}
			\abs{\alpha(sm+r) - \eta(m)} \leq e^{-(sm+r)\varepsilon}
		\end{equation*}
		for infinitely many values of $ m $, then for all but finitely many values of $ m $ and for $ p,q \in \ZZ $ with $ 0 < q < e^{(sm+r)\varepsilon} $ it holds
		\begin{equation*}
			\abs{\alpha(sm+r) - \frac{p}{q}} > \frac{1}{q^k} e^{-(sm+r)\varepsilon}.
		\end{equation*}
	\end{mythm}
	
	Before going on, let us give some remarks on the above theorem answering natural questions that might arise.
	
	\begin{myrem}
		We emphasize that the integer $ s $ in the above theorem does only depend on the polynomial $ f $.
		The integer $ k $ depends on the power sum $ G_n $, the polynomial $ f $ as well as on the integers $ s $ and $ r $, but it is independent of $ n $.
	\end{myrem}
	
	\begin{myrem}
		The assumption in Theorem \ref{thm:approxresult} that $ \varepsilon $ is small enough is a technical assumption needed in the proof.
		In practical applications one would show that for all $ \varepsilon > 0 $ there is no power sum $ \eta \in \QQ\E_{\ZZ} $ with non-negative characteristic roots such that
		\begin{equation*}
			\abs{\alpha(sm+r) - \eta(m)} \leq e^{-(sm+r)\varepsilon}
		\end{equation*}
		(cf. the proof of Corollary 3.2 in \cite{scremin-2006}).
	\end{myrem}
	
	\begin{myrem}
		The exponent $ k $ of the denominator $ q $ in the lower bound at the end of Theorem \ref{thm:approxresult} can in general be much larger than $ 2 $.
		But this are the costs we have to pay in order to get an explicit constant and, in particular, an explicit dependence of this ``constant'' on the index $ n $ if we compare this with Roth's theorem.
		Note that despite the possible larger size the integer $ k $ is a constant independent of $ \varepsilon $.
		Thus we get the new (weaker) lower bound
		\begin{equation*}
			\frac{1}{q^k} e^{-(sm+r)\varepsilon} > e^{-(sm+r)(k+1)\varepsilon} = e^{-(sm+r)\varepsilon'}
		\end{equation*}
		showing that the concrete value of $ k $ is not that important.
	\end{myrem}
	
	\begin{myrem}
		The theorem does only say something about small values of $ q $.
		But this is not really restrictive since the bound in the theorem is not significant for large values of $ q $.
		More precisely, there is a positive constant $ C $, depending on $ \alpha $, such that for $ q \geq e^{Cn} $ the lower bound from the theorem is weaker than the bound obtained directly from the inequality $ \abs{\theta} \geq \Hc(\theta)^{-1} $, where $ \Hc $ denotes the height function and $ \theta $ is a nonzero algebraic number (cf. the remarks in \cite{corvaja-zannier-2005}).
	\end{myrem}
	
	Theorem \ref{thm:approxresult} only deals with one power sum.
	In our second theorem we will consider the case when $ \alpha(n) $ is defined using several power sums.
	For this let $ f(x_0,\ldots,x_d,y) \in \QQ[x_0,\ldots,x_d,y] $ and $ G_n^{(0)}, \ldots, G_n^{(d)} \in \QQ\E_{\ZZ} $.
	Then the coefficients of $ f(G_n^{(0)}, \ldots, G_n^{(d)},y) $ as a polynomial in $ y $ are again elements of $ \QQ\E_{\ZZ} $.
	Hence without loss of generality we may assume that
	\begin{equation}
		\label{eq:finserted}
		f(G_n^{(0)}, \ldots, G_n^{(d)},y) = G_n^{(d)} y^d + \cdots + G_n^{(1)} y + G_n^{(0)}
	\end{equation}
	that is
	\begin{equation}
		\label{eq:smallf}
		f(x_0,\ldots,x_d,y) = x_d y^d + \cdots + x_1 y + x_0.
	\end{equation}
	Let $ c_1,\ldots,c_h \in \ZZ $ be all characteristic roots appearing in $ G_n^{(0)}, \ldots, G_n^{(d)} $.
	As above we may assume that $ c_1 > c_2 > \cdots > c_h > 0 $.
	Equation \eqref{eq:finserted} can now be rewritten as
	\begin{equation*}
		\widetilde{f}(c_1^n,\ldots,c_h^n,y) = a_d(c_1^n,\ldots,c_h^n) y^d + \cdots + a_0(c_1^n,\ldots,c_h^n)
	\end{equation*}
	with linear polynomials $ a_0,\ldots,a_d \in \QQ[x_1,\ldots,x_h] $.
	Here the further assumption $ c_h > 1 $ is no restriction as the $ a_i $ may have a constant part.
	Dividing by $ c_1^n $, we get that the solutions $ y $ of
	\begin{equation*}
		f(G_n^{(0)}, \ldots, G_n^{(d)},y) = 0
	\end{equation*}
	are those of
	\begin{equation*}
		F(g_1^n,\ldots,g_h^n,y) = 0
	\end{equation*}
	where
	\begin{equation}
		\label{eq:defofbigF}
		F(x_1,\ldots,x_h,y) = l_d(x_1,\ldots,x_h) y^d + \cdots + l_0(x_1,\ldots,x_h)
	\end{equation}
	for suitable linear polynomials $ l_0,\ldots,l_d \in \QQ[x_1,\ldots,x_h] $ and
	\begin{equation}
		\label{eq:defofgi}
		\set{g_1,\ldots,g_h} = \set{\frac{c_2}{c_1},\ldots,\frac{c_h}{c_1},\frac{1}{c_1}}
	\end{equation}
	with $ 1 > g_1 > g_2 > \cdots > g_h > 0 $.
	The polynomial $ F $ plays an important role in our second theorem which states the following:
	
	\begin{mythm}
		\label{thm:severallrs}
		Let $ f(x_0,\ldots,x_d,y) \in \QQ[x_0,\ldots,x_d,y] $ be as in \eqref{eq:smallf} and $ G_n^{(0)}, \ldots, G_n^{(d)} \in \QQ\E_{\ZZ} $ with non-negative characteristic roots and $ G_n^{(0)} \neq 0 $.
		Further construct $ F $ and $ g_1,\ldots,g_h $ as in \eqref{eq:defofbigF} and \eqref{eq:defofgi}, respectively.
		Let $ \alpha(n) $ be a solution $ y $ of
		\begin{equation*}
			f(G_n^{(0)}, \ldots, G_n^{(d)},y) = 0,
		\end{equation*}
		which is equivalent to saying
		\begin{equation*}
			F(g_1^n,\ldots,g_h^n,\alpha(n)) = 0.
		\end{equation*}
		Moreover, assume that $ l_d(0,\ldots,0) \neq 0 $ and that $ F(0,\ldots,0,y) $ has neither a multiple nor a rational zero as a polynomial in $ y $.
		
		Then there exists an integer $ k \geq 2 $ such that for all $ \varepsilon > 0 $ with $ \varepsilon < \frac{1}{2k} $ we have the following:
		For all but finitely many values of $ n $ and for $ p,q \in \ZZ $ with $ 0 < q < e^{n\varepsilon} $ it holds
		\begin{equation*}
			\abs{\alpha(n) - \frac{p}{q}} > \frac{1}{q^k} e^{-n\varepsilon}.
		\end{equation*}	
	\end{mythm}
	
	The assumption that $ F(0,\ldots,0,y) $ has no rational zero replaces the non-existence of $ \eta $ assumed in Theorem \ref{thm:approxresult}.
	
	\section{Preliminaries}
	
	We will need a suitable explicit version of Puiseux's theorem.
	For this we take Theorem 5 from \cite{fuchs-2003} which is a summary of some results of Coates \cite{coates-1970}.
	Assume that $ f(x,y) \in K[x,y] $ for a number field $ K $ is absolutely irreducible, with degree $ m $ in $ x $ and degree $ n $ in $ y $, and let $ f_0 \geq 2 $ be an upper bound for the absolute values of the conjugates of the coefficients of $ f $.
	Define $ N = \max \setb{n,m,3} $.
	Denoting by $ h(a) $ the maximum of the absolute values of the conjugates of an algebraic number $ a $, by $ \delta(a) $ the least positive rational integer such that $ \delta(a)a $ is an algebraic integer, and by $ \sigma(a) $ the maximum of $ \deg a, \delta(a), h(a) $, the statement is:
	
	\begin{mythm}[\textbf{Explicit version of Puiseux's theorem}]
		\label{thm:puiseux}
		Let $ F $ be an algebraic function field over an algebraically closed field $ \overline{K} $ of characteristic zero, given by $ f(x,y) = 0 $.
		Let $ \xi \in \overline{K} $ and $ A_i $ for $ 1 \leq i \leq r = r(\xi) $ be the valuations of $ F $ extending the valuation of $ \overline{K}(x) $ defined by $ x - \xi $, and let $ e_i $ be the ramification index of $ A_i $.
		We write
		\begin{equation*}
			y = \sum_{k=0}^{\infty} w_{ik} (x - \xi)^{k/e_i}
		\end{equation*}
		for the Puiseux expansion of $ y $ at $ A_i $.
		Then the coefficients $ w_{ik} $ ($ 1 \leq i \leq r $, $ k = 0,1,\ldots $) are algebraic numbers, and the number field $ K' $ obtained by adjoining $ \xi $ and these coefficients to $ K $ has degree at most $ (N \deg \xi)^N $ over $ K $.
		Further, $ K' $ is generated over $ K $ by $ \xi $ and $ w_{ik} $ ($ 1 \leq i \leq r $, $ 0 \leq k \leq 2N^4 $).
		Finally, there exists a positive rational integer $ \Delta $ such that $ \Delta^{k+1} w_{ik} $ ($ 1 \leq i \leq r $, $ k = 0,1,\ldots $) is an algebraic integer with
		\begin{equation*}
			\max \setb{\Delta^{k+1}, \Delta^{k+1} h(w_{ik})} \leq \Lambda^{k+1}
		\end{equation*}
		where $ \Lambda = (f_0 \sigma(\xi))^{\mu} $ for $  \mu = (N^4 n \deg \xi)^{3N^4} $.
		
		Let $ Q_i $ for $ 1 \leq i \leq r = r(\infty) $ be the valuations of $ F $ extending the valuation of $ \overline{K}(x) $ defined by $ 1/x $.
		Let $ e_i $ be the ramification index of $ Q_i $, and let
		\begin{equation*}
			y = \left( \frac{1}{x} \right)^{-m} \sum_{k=0}^{\infty} w_{ik} \left( \frac{1}{x} \right)^{k/e_i}
		\end{equation*}
		be the expansion of $ y $ at $ Q_i $.
		Then the coefficients are algebraic numbers, and the number field $ K' $ obtained by adjoining them to $ K $ has degree at most $ N^N $ over $ K $.
		Further, $ K' $ is generated over $ K $ by the $ w_{ik} $ ($ 1 \leq i \leq r $, $ 0 \leq k \leq 2N^4 $).
		Finally, there exists a positive rational integer $ \Delta $ such that $ \Delta^{k+1} w_{ik} $ ($ 1 \leq i \leq r $, $ k = 0,1,\ldots $) is an algebraic integer with
		\begin{equation*}
			\max \setb{\Delta^{k+1}, \Delta^{k+1} h(w_{ik})} \leq \Lambda^{k+1}
		\end{equation*}
		where $ \Lambda = f_0^{\mu} $ for $  \mu = (N^4 n)^{3N^4} $.
	\end{mythm}
	
	Puiseux's theorem will be used in the proof of our first approximation result. The second one will be proven by means of a suitable version of the Implicit function theorem.
	For more information about this and other versions of the Implicit function theorem we refer to \cite{krantz-parks-2002-1} and \cite{krantz-parks-2002-2}.
	In the formulation of the theorem we use for a multiindex $ \tau = (\tau_1,\ldots,\tau_r) \in \NN_0^r $, where $ \NN_0 $ denotes the set of non-negative integers, the notation
	\begin{equation*}
		\abs{\tau} = \tau_1 + \cdots + \tau_r
	\end{equation*}
	and write $ 0 $ as a shortcut for $ (0,\ldots,0) $.
	
	\begin{mythm}[\textbf{Implicit function theorem}]
		\label{thm:implicitfunc}
		Suppose the power series
		\begin{equation*}
			F(x_1,\ldots,x_r,y) = \sum_{\abs{\tau} \geq 0, k \geq 0} a_{\tau,k} x_1^{\tau_1} \cdots x_r^{\tau_r} y^k
		\end{equation*}
		is absolutely convergent for $ \abs{x_1} + \cdots + \abs{x_r} \leq R_1 $, $ \abs{y} \leq R_2 $. If
		\begin{equation*}
			a_{0,0} = 0 \text{ and } a_{0,1} \neq 0
		\end{equation*}
		then there exist $ r_0 > 0 $ and a power series
		\begin{equation}
			\label{eq:implfuncpowser}
			f(x_1,\ldots,x_r) = \sum_{\abs{\tau} > 0} A_{\tau} x_1^{\tau_1} \cdots x_r^{\tau_r}
		\end{equation}
		such that \eqref{eq:implfuncpowser} is absolutely convergent for $ \abs{x_1} + \cdots + \abs{x_r} \leq r_0 $ and
		\begin{equation*}
			F(x_1,\ldots,x_r,f(x_1,\ldots,x_r)) = 0.
		\end{equation*}
		Moreover, if the coefficients of $ F $ are algebraic, then the coefficients of $ f $ are also algebraic.
	\end{mythm}
	
	This version of the Implicit function theorem was used as well in \cite{fuchs-heintze-p5} where also polynomials with linear recurrences as coefficients are considered. As done there, we emphasize here again that the statement holds in a more general form:
	Suppose that $ F(x_1,\ldots,x_r,y) $ converges absolutely for $ \abs{x_1} + \cdots + \abs{x_r} \leq R_1 $ and that $ \abs{y-y_0} \leq R_2 $ for some $ y_0 \in \overline{\QQ} $ with $ F(0,\ldots,0,y_0) = 0 $. Then under the assumption that
	\begin{equation*}
		\frac{\partial F}{\partial y} (0,\ldots,0,y_0) \neq 0,
	\end{equation*}
	the conclusion is that there exists a power series
	\begin{equation*}
		f(x_1,\ldots,x_r) = \sum_{\abs{\tau} \geq 0} A_{\tau} x_1^{\tau_1} \cdots x_r^{\tau_r}
	\end{equation*}
	for which the same as above holds.
	
	Because we will need this later on, let us now point out that we also get an upper bound for the coefficients $ A_{\tau} $ of the power series $ f(x_1,\ldots,x_r) $ from the Implicit function theorem.
	Since $ f(x_1,\ldots,x_r) $ is absolutely convergent for $ \abs{x_1} + \cdots + \abs{x_r} \leq r_0 $, we can choose
	\begin{equation*}
		x_1 = \cdots = x_r = \frac{r_0}{r}
	\end{equation*}
	and get that the series
	\begin{equation*}
		\sum_{\abs{\tau} \geq 0} \abs{A_{\tau}} \left( \frac{r_0}{r} \right)^{\abs{\tau}}
	\end{equation*}
	is convergent.
	Hence for all but finitely many $ \tau $ we have
	\begin{equation*}
		\abs{A_{\tau}} \left( \frac{r_0}{r} \right)^{\abs{\tau}} \leq 1
	\end{equation*}
	and thus
	\begin{equation*}
		\abs{A_{\tau}} \leq \left( \frac{r}{r_0} \right)^{\abs{\tau}}.
	\end{equation*}
	Then there exists a constant $ \lambda > 0 $ such that for all $ \tau $ with $ \abs{\tau} > 0 $ we have
	\begin{equation*}
		\abs{A_{\tau}} \leq \lambda^{\abs{\tau}}.
	\end{equation*}
	
	Moreover, in our proofs we make use of a suitable version of the subspace theorem.
	We are going to apply the following version due to Schlickewei which is used in \cite{scremin-2006} as well and can also be found in \cite{bilu-2008}:
	
	\begin{mythm}[\textbf{Subspace theorem}]
		\label{thm:subspace}
		Let $ S $ be a finite set of absolute values of $ \QQ $, including the archimedean one and normalized in the usual way.
		Extend each $ \nu \in S $ to $ \overline{\QQ} $ in some way.
		For $ \nu \in S $ let $ L_{1,\nu}, \ldots, L_{n,\nu} $ be $ n $ linearly independent linear forms in $ n $ variables with algebraic coefficients.
		Finally, let $ \delta > 0 $.
		Then the solutions $ \uline{x} := (x_1,\ldots,x_n) \in \ZZ^n $ to the inequality
		\begin{equation*}
			\prod_{\nu \in S} \prod_{i=1}^{n} \abs{L_{i,\nu}(\uline{x})}_{\nu} < \left( \max_{1 \leq i \leq n} \abs{x_i} \right)^{-\delta}
		\end{equation*}
		are contained in finitely many proper subspaces of $ \QQ^n $.
	\end{mythm}
	
	Finally, we need the following auxiliary result which is proven as Lemma 2 in \cite{corvaja-zannier-2005} and also used as Lemma 4.2 in \cite{scremin-2006}:
	
	\begin{mylemma}
		\label{lem:qversusz}
		Let $ \xi \in \QQ\E_{\QQ} $ and let $ D $ be the minimal positive integer such that $ D^n \xi(n) \in \QQ\E_{\ZZ} $.
		Then, for every $ \varepsilon > 0 $, there are only finitely many $ n \in \NN $ such that the denominator of $ \xi(n) $ is smaller than $ D^n e^{-n\varepsilon} $.
	\end{mylemma}
	
	As usual we write $ g(m) = O(h(m)) $ and $ g(m) \ll h(m) $ if there is a positive constant $ C $ such that $ \abs{g(m)} \leq C \abs{h(m)} $ and $ g(m) \leq C h(m) $, respectively.
	
	\section{Proofs}
	
	The proof of our Theorem \ref{thm:approxresult} has the same strategy as the proof of Theorem 3.1 in \cite{scremin-2006}.
	As a first step we will prove the following:
	
	\begin{mylemma}
		\label{lem:approxps}
		Let $ f(x,y) $, $ G_n $ and $ \alpha(n) $ be as in Theorem \ref{thm:approxresult}.
		Fix $ t > 0 $.
		Then there exists a positive integer $ s $ such that for each fixed $ r \in \set{0,1,\ldots,s-1} $ there is a power sum $ \eta_r \in \overline{\QQ}\E_{\QQ} $ satisfying
		\begin{equation*}
			\abs{\alpha(sm+r) - \eta_r(m)} \ll t^{sm}.
		\end{equation*}
	\end{mylemma}
	
	\begin{proof}
		Using Theorem \ref{thm:puiseux}, we can express the solution $ y = y(x) $ of $ f(x,y) = 0 $ in its Puiseux expansion at $ x = \infty $.
		Thus we get
		\begin{equation*}
			y = \sum_{k=v}^{\infty} a_k x^{-k/e}
		\end{equation*}
		as well as the upper bound
		\begin{equation*}
			\abs{a_k} \leq \lambda^k
		\end{equation*}
		for $ k > 0 $ and a positive constant $ \lambda $.
		We put $ s := e $.
		Now we replace $ x $ by $ G_n $ and $ y $ by $ \alpha(n) $ to get the series expansion
		\begin{equation}
			\label{eq:alphaexp}
			\alpha(n) = \sum_{k=v}^{\infty} a_k G_n^{-k/s}.
		\end{equation}
		
		In order to construct the power sum $ \eta_r $ with the required approximation property, we will cut this series \eqref{eq:alphaexp} at a suitable point.
		For an integer $ K > 0 $ we can write
		\begin{equation*}
			\alpha(n) = \sum_{k=v}^{K} a_k G_n^{-k/s} + \sum_{k=K+1}^{\infty} a_k G_n^{-k/s}
		\end{equation*}
		and get the bound
		\begin{align*}
			\abs{\sum_{k=K+1}^{\infty} a_k G_n^{-k/s}} &\leq \sum_{k=K+1}^{\infty} \abs{a_k} \abs{G_n}^{-k/s} \\
			&\leq \sum_{k=K+1}^{\infty} \lambda^k \abs{G_n}^{-k/s} \\
			&= \sum_{k=K+1}^{\infty} \left( \frac{\lambda}{\abs{G_n}^{1/s}} \right)^k \\
			&= \left( \frac{\lambda}{\abs{G_n}^{1/s}} \right)^{K+1} \cdot \frac{1}{1 - \frac{\lambda}{\abs{G_n}^{1/s}}} \\
			&\leq 2 \left( \frac{\lambda}{\abs{G_n}^{1/s}} \right)^{K+1}
		\end{align*}
		for $ n \geq n_0 $ since $ \abs{G_n} $ goes to infinity as $ n $ does.
		Therefore we have
		\begin{equation*}
			\alpha(n) = \sum_{k=v}^{K} a_k G_n^{-k/s} + O\left( \left( \frac{\lambda}{\abs{G_n}^{1/s}} \right)^{K+1} \right)
		\end{equation*}
		for $ n \geq n_0 $.
		
		Consider now the arithmetic progression $ n = sm+r $ for $ s $ from above and fixed $ r \in \set{0,1,\ldots,s-1} $.
		Note that then bounds for $ n $ correspond to bounds for $ m $ and vice versa.
		For $ n \geq n_1 \geq n_0 $ we get
		\begin{equation*}
			\abs{G_n} \geq \frac{\abs{b_1}}{2} c_1^n
		\end{equation*}
		as well as
		\begin{equation*}
			\frac{\lambda}{\abs{G_n}^{1/s}} \leq \lambda \left( \frac{2}{\abs{b_1}} \right)^{1/s} c_1^{-n/s}
		\end{equation*}
		and
		\begin{equation*}
			\frac{\lambda}{\abs{G_{sm+r}}^{1/s}} \leq \underbrace{\lambda \left( \frac{2}{\abs{b_1}} \right)^{1/s} c_1^{-r/s}}_{=:g_r}\ \cdot\ c_1^{-m} = g_r c_1^{-m}.
		\end{equation*}
		At this point we choose $ K > 0 $ such that
		\begin{equation*}
			c_1^{-(K+1)/s} < t.
		\end{equation*}
		Thus, from here on $ K $ is fixed.
		Going on with the bound above yields
		\begin{equation*}
			\left( \frac{\lambda}{\abs{G_{sm+r}}^{1/s}} \right)^{K+1} \leq g_r^{K+1} c_1^{-(K+1)m} < g_r^{K+1} t^{sm}.
		\end{equation*}
		Hence we end up with
		\begin{equation}
			\label{eq:alphacutone}
			\alpha(sm+r) = \sum_{k=v}^{K} a_k G_{sm+r}^{-k/s} + O(t^{sm})
		\end{equation}
		for $ n \geq n_1 $.
		
		In the next step we take a closer look at $ G_{sm+r}^{-k/s} $ for $ k \in \set{v,v+1,\ldots,K} $.
		For this purpose we write
		\begin{equation*}
			G_n = b_1 c_1^n (1 + \sigma(n))
		\end{equation*}
		with
		\begin{equation*}
			\sigma(n) = \sum_{j=2}^{h} \frac{b_j}{b_1} \left( \frac{c_j}{c_1} \right)^n.
		\end{equation*}
		By the binomial series expansion we get
		\begin{equation*}
			G_n^{-k/s} = b_1^{-k/s} c_1^{-kn/s} \sum_{\ell=0}^{\infty} \bino{-k/s}{\ell} (\sigma(n))^{\ell}.
		\end{equation*}
		Again we aim for cutting this series.
		It is easy to see that there is a positive constant $ B $ such that for each $ k \in \set{v,v+1,\ldots,K} $ and $ \ell = 0,1,\ldots $ the bound
		\begin{equation*}
			\abs{\bino{-k/s}{\ell}} \leq B^{\ell}
		\end{equation*}
		is valid.
		Therefore, for a given value of $ L $, we get the upper bound
		\begin{align*}
			\abs{\sum_{\ell=L+1}^{\infty} \bino{-k/s}{\ell} (\sigma(n))^{\ell}} &\leq \sum_{\ell=L+1}^{\infty} \abs{\bino{-k/s}{\ell}} \abs{\sigma(n)}^{\ell} \\
			&\leq \sum_{\ell=L+1}^{\infty} \abs{B \sigma(n)}^{\ell} \\
			&= \abs{B \sigma(n)}^{L+1} \cdot \frac{1}{1 - \abs{B \sigma(n)}} \\
			&\leq 2\abs{B \sigma(n)}^{L+1}
		\end{align*}
		for $ n \geq n_2 \geq n_1 $ since $ \sigma(n) $ goes to $ 0 $ as $ n $ goes to infinity.
		Note that
		\begin{equation*}
			\sigma(n) = O\left( \left( \frac{c_2}{c_1} \right)^n \right).
		\end{equation*}
		This gives us now
		\begin{equation*}
			G_{sm+r}^{-k/s} = (b_1 c_1^r)^{-k/s} c_1^{-km} \sum_{\ell=0}^{L} \bino{-k/s}{\ell} (\sigma(sm+r))^{\ell} + O\left( c_1^{-km} \left( \frac{c_2}{c_1} \right)^{sm(L+1)} \right).
		\end{equation*}
		At this point we choose $ L \geq 0 $ such that for each $ k \in \set{v,v+1,\ldots,K} $ we have
		\begin{equation*}
			c_1^{-k/s} \left( \frac{c_2}{c_1} \right)^{L+1} < t.
		\end{equation*}
		Note that in fact this holds for all $ k \in \set{v,v+1,\ldots,K} $ if (and only if) it holds for $ k = v $.
		So $ L $ is also fixed now.
		Moreover, we get
		\begin{equation*}
			c_1^{-km} \left( \frac{c_2}{c_1} \right)^{sm(L+1)} = \left( c_1^{-k/s} \left( \frac{c_2}{c_1} \right)^{L+1} \right)^{sm} < t^{sm}
		\end{equation*}
		and thus
		\begin{equation*}
			G_{sm+r}^{-k/s} = (b_1 c_1^r)^{-k/s} c_1^{-km} \sum_{\ell=0}^{L} \bino{-k/s}{\ell} (\sigma(sm+r))^{\ell} + O\left( t^{sm} \right)
		\end{equation*}
		for $ n \geq n_2 $.
		
		Finally, we put the last expression into equation \eqref{eq:alphacutone} which yields
		\begin{equation*}
			\alpha(sm+r) = \sum_{k=v}^{K} a_k (b_1 c_1^r)^{-k/s} c_1^{-km} \sum_{\ell=0}^{L} \bino{-k/s}{\ell} (\sigma(sm+r))^{\ell} + O(t^{sm})
		\end{equation*}
		for $ n \geq n_2 $.
		Hence we define $ \eta_r \in \overline{\QQ}\E_{\QQ} $ by
		\begin{equation*}
			\eta_r(m) := \sum_{k=v}^{K} a_k (b_1 c_1^r)^{-k/s} c_1^{-km} \sum_{\ell=0}^{L} \bino{-k/s}{\ell} (\sigma(sm+r))^{\ell}
		\end{equation*}
		which satisfies
		\begin{equation}
			\label{eq:finalbound}
			\abs{\alpha(sm+r) - \eta_r(m)} \ll t^{sm}
		\end{equation}
		for $ n \geq n_2 $.
		By taking a new constant which fits also for the finitely many smaller values of $ n $, the bound \eqref{eq:finalbound} holds for all $ n \in \NN $.
		This proves the lemma.
	\end{proof}
	
	Having finished all preparations needed, we can now prove our first approximation theorem.
	
	\begin{proof}[Proof of Theorem \ref{thm:approxresult}]
		Put $ t = 1/9 $.
		Lemma \ref{lem:approxps} gives a positive integer $ s $ and for each $ r \in \set{0,1,\ldots,s-1} $, which we fix from here on, a power sum $ \eta_r \in \overline{\QQ}\E_{\QQ} $ such that
		\begin{equation*}
			\abs{\alpha(sm+r) - \eta_r(m)} \ll t^{sm}.
		\end{equation*}
		The construction of $ \eta_r $ shows that we can write
		\begin{equation*}
			\eta_r(m) = w_1 d_1^m + \cdots + w_H d_H^m
		\end{equation*}
		for $ H \geq 1 $ (since $ \alpha \neq 0 $ by the irreducibility assumption) as well as $ d_1,\ldots,d_H \in \QQ $ with $ d_1 > d_2 > \cdots > d_H > 0 $ and $ w_1,\ldots,w_H \in \overline{\QQ}^* $.
		Define $ k := H+1 \geq 2 $ and choose $ \varepsilon > 0 $ smaller than
		\begin{equation*}
			\min \setb{\frac{1}{2(s+2)}, \frac{1}{2k}}.
		\end{equation*}
		
		Assume indirectly that there are infinitely many values of $ m $ with corresponding $ (p,q) \in \ZZ^2 $ satisfying $ 0 < q < e^{(sm+r)\varepsilon} $ such that
		\begin{equation}
			\label{eq:indirassump}
			\abs{\alpha(sm+r) - \frac{p}{q}} \leq \frac{1}{q^k} e^{-(sm+r)\varepsilon}.
		\end{equation}
		
		We aim to apply the subspace theorem.
		Therefore fix a finite set $ S $ of absolute values of $ \QQ $ containing the archimedean one and such that $ d_1,\ldots,d_H $ are $ S $-units (i.e. their $ \nu $-valuation is $ 1 $ for all $ \nu \notin S $).
		Moreover we define for each $ \nu \in S $ the $ H+1 $ linear forms $ L_{0,\nu}, \ldots, L_{H,\nu} $ in the $ H+1 $ variables $ X_0, \ldots, X_H $ as follows.
		For $ i=0,\ldots,H $ and $ \nu \in S $ with either $ i \neq 0 $ or $ \nu \neq \infty $ we put
		\begin{equation*}
			L_{i,\nu} := X_i
		\end{equation*}
		and for the remaining one
		\begin{equation*}
			L_{0,\infty} := X_0 - w_1 X_1 - \cdots - w_H X_H.
		\end{equation*}
		Obviously for each $ \nu \in S $ the linear forms $ L_{0,\nu}, \ldots, L_{H,\nu} $ are linearly independent and all linear forms have algebraic coefficients.
		Let $ d $ be the smallest positive rational integer such that $ dd_i \in \ZZ $ for all $ i=1,\ldots,H $.
		Then $ d $ is also an $ S $-unit.
		We set $ e_i := dd_i $ for $ i=1,\ldots,H $.
		So $ e_i \in \ZZ $ for $ i=1,\ldots,H $.
		Now we consider the vectors
		\begin{equation*}
			\uline{x} = \uline{x}(m,p,q) := (pd^m, qe_1^m, qe_2^m, \ldots, qe_H^m) \in \ZZ^{H+1}.
		\end{equation*}
		
		In the next step, we proceed with taking a closer look at the double product appearing in the subspace theorem.
		Our first goal is to determine an upper bound for
		\begin{equation*}
			\prod_{\nu \in S} \prod_{i=0}^{H} \abs{L_{i,\nu}(\uline{x})}_{\nu}.
		\end{equation*}
		This double product can be splitted into
		\begin{equation*}
			\left( \prod_{i=1}^{H} \prod_{\nu \in S} \abs{L_{i,\nu}(\uline{x})}_{\nu} \right) \cdot \left( \prod_{\nu \in S \setminus \set{\infty}} \abs{L_{0,\nu}(\uline{x})}_{\nu} \right) \cdot \abs{L_{0,\infty}(\uline{x})}.
		\end{equation*}
		Since $ d $ as well as $ d_i $ and $ e_i $ for $ i=1,\ldots,H $ are $ S $-units, we have
		\begin{equation*}
			\prod_{\nu \in S} \abs{L_{i,\nu}(\uline{x})}_{\nu} = \prod_{\nu \in S} \abs{qe_i^m}_{\nu} = \prod_{\nu \in S} \abs{q}_{\nu} \leq q
		\end{equation*}
		for each $ i=1,\ldots,H $ and thus
		\begin{equation*}
			\prod_{i=1}^{H} \prod_{\nu \in S} \abs{L_{i,\nu}(\uline{x})}_{\nu} \leq q^H.
		\end{equation*}
		Furthermore, for the second product we get
		\begin{align*}
			\prod_{\nu \in S \setminus \set{\infty}} \abs{L_{0,\nu}(\uline{x})}_{\nu} &= \prod_{\nu \in S \setminus \set{\infty}} \abs{pd^m}_{\nu} \\
			&= \prod_{\nu \in S \setminus \set{\infty}} \abs{p}_{\nu} \cdot \prod_{\nu \in S \setminus \set{\infty}} \abs{d^m}_{\nu} \\
			&\leq 1 \cdot d^{-m} = d^{-m}.
		\end{align*}
		Finally, we have
		\begin{align*}
			\abs{L_{0,\infty}(\uline{x})} &= \abs{pd^m - w_1 q e_1^m - w_2 q e_2^m - \cdots - w_H q e_H^m} \\
			&= d^m \abs{p - q (w_1 d_1^m + \cdots + w_H d_H^m)} \\
			&= q d^m \abs{\frac{p}{q} - \eta_r(m)}.
		\end{align*}
		Putting all these bounds together yields
		\begin{equation}
			\label{eq:sstboundone}
			\prod_{\nu \in S} \prod_{i=0}^{H} \abs{L_{i,\nu}(\uline{x})}_{\nu} \leq q^{H+1} \abs{\eta_r(m) - \frac{p}{q}}.
		\end{equation}
		
		Since $ q < e^{(sm+r)\varepsilon} $, we have $ q^{-k} > e^{-k(sm+r)\varepsilon} $ and consequently the bound $ q^{-k} e^{-(sm+r)\varepsilon} > e^{-(k+1)(sm+r)\varepsilon} $.
		Recalling $ \varepsilon < \frac{1}{2k} $ as well as $ k \geq 2 $ and $ t = 1/9 $ gives us
		\begin{align}
			q^{-k} e^{-(sm+r)\varepsilon} &> e^{-(k+1)(sm+r)\varepsilon} > e^{-(sm+r) \frac{k+1}{2k}} \nonumber \\
			\label{eq:boundsqrtt}
			&= e^{-(sm+r) \left( \frac{1}{2} + \frac{1}{2k} \right)} > e^{-(sm+r) \left( \frac{1}{2} + \frac{1}{4} \right)} \\
			&= \left( e^{-3/4} \right)^{sm+r} > \left( \frac{1}{3} \right)^{sm+r} = \left( \sqrt{t} \right)^{sm+r}. \nonumber
		\end{align}
		Now Lemma \ref{lem:approxps} gave us a constant $ l > 0 $ such that
		\begin{align*}
			\abs{\eta_r(m) - \frac{p}{q}} &\leq \abs{\eta_r(m) - \alpha(sm+r)} + \abs{\alpha(sm+r) - \frac{p}{q}} \\
			&\leq l t^{sm+r} + \frac{1}{q^k} e^{-(sm+r)\varepsilon} \\
			&= \left( l \left( \sqrt{t} \right)^{sm+r} \right) \cdot \left( \sqrt{t} \right)^{sm+r} + \frac{1}{q^k} e^{-(sm+r)\varepsilon} \\
			&\leq \frac{2}{q^k} e^{-(sm+r)\varepsilon}
		\end{align*}
		for $ m $ large enough and by using the bounds \eqref{eq:indirassump} and \eqref{eq:boundsqrtt}.
		As $ 2 \leq e^{(sm+r) \varepsilon / 2} $ for $ m $ large enough, we can update the bound \eqref{eq:sstboundone} to
		\begin{align}
			\prod_{\nu \in S} \prod_{i=0}^{H} \abs{L_{i,\nu}(\uline{x})}_{\nu} &\leq q^{H+1} \abs{\eta_r(m) - \frac{p}{q}} \nonumber \\
			\label{eq:sstboundtwo}
			&\leq 2 q^{H+1-k} e^{-(sm+r) \varepsilon} \\
			&\leq e^{-(sm+r) \varepsilon / 2}. \nonumber
		\end{align}
		
		Having an upper bound for the double product, the next goal is to bound the maximum appearing in the inequality in the subspace theorem.
		We have already deduced the bound
		\begin{equation*}
			\abs{\eta_r(m) - \frac{p}{q}} \leq \frac{2}{q^k} e^{-(sm+r)\varepsilon}.
		\end{equation*}
		From this we can infer that
		\begin{equation*}
			\abs{\frac{p/q}{d_1^m}} \leq 2 \abs{w_1}
		\end{equation*}
		for large $ m $.
		Consequently, we have
		\begin{equation*}
			\abs{pd^m} \leq 2 \abs{w_1} q e_1^m
		\end{equation*}
		for still infinitely many $ m $.
		
		Hence, again for large values of $ m $, we can bound the appearing maximum in the following way:
		\begin{align*}
			\max_{0 \leq i \leq H} \abs{x_i} &= \max \setb{\abs{pd^m}, qe_1^m, \ldots, qe_H^m} \\
			&= \max \setb{\abs{pd^m}, qe_1^m} \\
			&\leq (1 + 2 \abs{w_1}) qe_1^m \\
			&\leq (1 + 2 \abs{w_1}) e^{(sm+r)\varepsilon} e_1^m \\
			&\leq e^{(sm+r)\varepsilon} e^{(sm+r)\varepsilon} e_1^m \\
			&\leq \left( e^{2\varepsilon} e_1 \right)^{sm+r}.
		\end{align*}
		Now choose $ \delta > 0 $ small enough such that
		\begin{equation*}
			\delta < \frac{\varepsilon}{2 \log \left( e^{2\varepsilon} e_1 \right)}
		\end{equation*}
		which yields
		\begin{align*}
			\left( \max_{0 \leq i \leq H} \abs{x_i} \right)^{\delta} &\leq \left( e^{2\varepsilon} e_1 \right)^{(sm+r) \delta} = \left( \left( e^{2\varepsilon} e_1 \right)^{\delta} \right)^{sm+r} \\
			&< \left( e^{\varepsilon / 2} \right)^{sm+r} = e^{(sm+r) \varepsilon / 2}.
		\end{align*}
		Together with inequality \eqref{eq:sstboundtwo} we end up with
		\begin{equation*}
			\prod_{\nu \in S} \prod_{i=0}^{H} \abs{L_{i,\nu}(\uline{x})}_{\nu} \leq e^{-(sm+r) \varepsilon / 2} < \left( \max_{0 \leq i \leq H} \abs{x_i} \right)^{-\delta}.
		\end{equation*}
		
		Since this inequality holds for infinitely many vectors $ \uline{x} = \uline{x}(m,p,q) $, Theorem \ref{thm:subspace} implies that infinitely many vectors $ \uline{x} $ are contained in a fixed proper subspace of $ \QQ^{H+1} $.
		Let
		\begin{equation*}
			z_0 X_0 - z_1 X_1 - \cdots - z_H X_H = 0
		\end{equation*}
		with $ z_0, z_1, \ldots, z_H \in \QQ $ be the defining equation of this subspace.
		Thus we have
		\begin{equation}
			\label{eq:relation}
			z_0 p d^m - z_1 q e_1^m - \cdots - z_H q e_H^m = 0
		\end{equation}
		for infinitely many triples $ (m,p,q) $.
		
		If $ z_0 = 0 $, then we would have
		\begin{equation*}
			z_1 e_1^m + \cdots + z_H e_H^m = 0
		\end{equation*}
		for infinitely many $ m $.
		Since $ e_1, \ldots, e_H $ are distinct positive integers, this ends up in the contradiction $ z_1 = \cdots = z_H = 0 $.
		
		Therefore we can assume that $ z_0 = 1 $.
		Rewriting equation \eqref{eq:relation} gives us the equality
		\begin{equation}
			\label{eq:betam}
			\frac{p}{q} = z_1 d_1^m + \cdots + z_H d_H^m =: \beta(m).
		\end{equation}
		Note that $ \beta \in \QQ\E_{\QQ} $.
		We will show that in fact $ d_1, \ldots, d_H $ must be integers if the corresponding coefficient in $ \beta $ is non-zero.
		Assume the contrary, i.e. there exists an index $ i $ such that $ d_i \in \QQ \setminus \ZZ $.
		Then the smallest positive integer $ D $ with the property that $ D^m \beta(m) \in \QQ\E_{\ZZ} $ satisfies $ D \geq 2 $.
		By Lemma \ref{lem:qversusz}, for all but finitely many $ m $ the denominator of $ \beta(m) $ is bounded from below by $ D^m e^{-m\varepsilon} $.
		Thus for infinitely many $ m $ we have
		\begin{equation*}
			q \geq D^m e^{-m\varepsilon} \geq 2^m e^{-m\varepsilon}.
		\end{equation*}
		Comparing this with the upper bound
		\begin{equation*}
			q < e^{(sm+r) \varepsilon} \leq e^{(s+1)m \varepsilon}
		\end{equation*}
		gives
		\begin{equation*}
			e^{(s+1) \varepsilon} > q^{1/m} \geq 2 e^{-\varepsilon}
		\end{equation*}
		and, since $ \varepsilon < \frac{1}{2(s+2)} $, also
		\begin{equation*}
			2 < e^{(s+2) \varepsilon} < e^{1/2}.
		\end{equation*}
		This is a contradiction, proving that $ d_i \in \ZZ $ for $ i=1,\ldots,H $ if the corresponding coefficient is not zero.
		Hence we have $ \beta \in \QQ\E_{\ZZ} $.
		
		Now we insert equation \eqref{eq:betam} into equation \eqref{eq:indirassump}.
		The result is that for infinitely many $ m $ we have
		\begin{equation*}
			\abs{\alpha(sm+r) - \beta(m)} \leq \frac{1}{q^k} e^{-(sm+r)\varepsilon} \leq e^{-(sm+r)\varepsilon}.
		\end{equation*}
		Since $ \beta \in \QQ\E_{\ZZ} $ has non-negative characteristic roots, this contradicts the assumption in the theorem.
		Thus the theorem is proven.
	\end{proof}
	
	For the proof of Theorem \ref{thm:severallrs}, which will be done in an analogous way as for Theorem \ref{thm:approxresult}, we start again with proving a lemma, similar to Lemma \ref{lem:approxps}.
	In the proof of this lemma we use some ideas also appearing in \cite{fuchs-heintze-p5} to prepare the application of the Implicit function theorem.
	
	\begin{mylemma}
		\label{lem:approximp}
		Let $ f, G_n^{(0)}, \ldots, G_n^{(d)}, F, g_1, \ldots, g_h $ and $ \alpha(n) $ be as in Theorem \ref{thm:severallrs}.
		Fix $ t > 0 $.
		Then there exists a power sum $ \eta \in \overline{\QQ}\E_{\QQ} $ satisfying
		\begin{equation*}
			\abs{\alpha(n) - \eta(n)} \ll t^n.
		\end{equation*}
	\end{mylemma}
	
	\begin{proof}
		By definition we have
		\begin{equation*}
			F(g_1^n,\ldots,g_h^n,\alpha(n)) = 0
		\end{equation*}
		which can be written as
		\begin{equation*}
			l_d(g_1^n,\ldots,g_h^n) \alpha(n)^d + \cdots + l_0(g_1^n,\ldots,g_h^n) = 0.
		\end{equation*}
		Isolating one $ \alpha(n) $ gives
		\begin{equation*}
			\alpha(n) = - \frac{l_{d-1}(g_1^n,\ldots,g_h^n) + \cdots + l_0(g_1^n,\ldots,g_h^n) \alpha(n)^{-(d-1)}}{l_d(g_1^n,\ldots,g_h^n)}
		\end{equation*}
		which implies
		\begin{equation*}
			\abs{\alpha(n)} \leq \max \setb{1, \frac{\abs{l_{d-1}(g_1^n,\ldots,g_h^n)} + \cdots + \abs{l_0(g_1^n,\ldots,g_h^n)}}{\abs{l_d(g_1^n,\ldots,g_h^n)}}}.
		\end{equation*}
		Since $ 0 < g_i < 1 $ for all $ i=1,\ldots,h $ and $ l_d(0,\ldots,0) \neq 0 $, we get that $ \alpha(n) $ is bounded.
		This boundedness together with
		\begin{align*}
			\abs{F(0,\ldots,0,\alpha(n))} &= \abs{F(0,\ldots,0,\alpha(n)) - F(g_1^n,\ldots,g_h^n,\alpha(n))} \\
			&\leq \sum_{i=0}^{d} \abs{l_i(0,\ldots,0) - l_i(g_1^n,\ldots,g_h^n)} \cdot \abs{\alpha(n)}^i \quad \overset{n \rightarrow \infty}{\longrightarrow} 0
		\end{align*}
		yields that $ F(0,\ldots,0,\alpha(n)) $ goes to $ 0 $ as $ n $ goes to infinity.
		Thus, for $ n $ large enough, $ \alpha(n) $ lies in an arbitrary small neighborhood of a zero of $ F(0,\ldots,0,y) $.
		The condition that $ F(0,\ldots,0,y) $ has no multiple zero as a polynomial in $ y $ is equivalent to
		\begin{equation*}
			\frac{\partial F}{\partial y} (0,\ldots,0,y_0) \neq 0
		\end{equation*}
		for any $ y_0 $ satisfying $ F(0,\ldots,0,y_0) = 0 $.
		Hence we can apply the Implicit function theorem \ref{thm:implicitfunc} including the appended remarks.
		This gives us a power series
		\begin{equation*}
			y(x_1,\ldots,x_h) = \sum_{\abs{\tau} \geq 0} A_{\tau} x_1^{\tau_1} \cdots x_h^{\tau_h}
		\end{equation*}
		with algebraic coefficients and a constant $ \lambda > 0 $ such that
		\begin{equation*}
			\abs{A_{\tau}} \leq \lambda^{\abs{\tau}}
		\end{equation*}
		for $ \abs{\tau} > 0 $ and
		\begin{equation*}
			\alpha(n) = y(g_1^n,\ldots,g_h^n)
		\end{equation*}
		for $ n $ large enough.
		So we have
		\begin{equation}
			\label{eq:alphacuttwo}
			\alpha(n) = \sum_{\abs{\tau} \geq 0} A_{\tau} \left( g_1^n \right)^{\tau_1} \cdots \left( g_h^n \right)^{\tau_h}
		\end{equation}
		with $ 0 < g_h < \cdots < g_2 < g_1 < 1 $, $ g_i \in \QQ $ for $ i=1,\ldots,h $ if $ n $ is large.
		
		In what follows we will need an upper bound for the number of vectors $ \tau = (\tau_1,\ldots,\tau_h) \in \NN_0^h $ satisfying $ \abs{\tau} = k $ for a given positive integer $ k $.
		One can think of this combinatoric question as counting the number of possibilities for distributing $ k $ equal balls into $ h $ different urns.
		It is well known that this number is given by
		\begin{equation*}
			\bino{k+h-1}{k} = \prod_{i=1}^{k} \frac{i+h-1}{i} = \prod_{i=1}^{k} \left( 1 + \frac{h-1}{i} \right) \leq h^k.
		\end{equation*}
		Hence there are no more than $ h^k $ vectors $ \tau $ with $ \abs{\tau} = k $.
		
		For a given integer $ K \geq 0 $ we can split equation \eqref{eq:alphacuttwo} into
		\begin{equation*}
			\alpha(n) = \sum_{\abs{\tau} = 0}^{K} A_{\tau} \left( g_1^n \right)^{\tau_1} \cdots \left( g_h^n \right)^{\tau_h} + \sum_{\abs{\tau} = K+1}^{\infty} A_{\tau} \left( g_1^n \right)^{\tau_1} \cdots \left( g_h^n \right)^{\tau_h}.
		\end{equation*}
		The second sum in this equation can be bounded by
		\begin{align*}
			\abs{\sum_{\abs{\tau} = K+1}^{\infty} A_{\tau} \left( g_1^n \right)^{\tau_1} \cdots \left( g_h^n \right)^{\tau_h}} &\leq \sum_{\abs{\tau} = K+1}^{\infty} \abs{A_{\tau}} \left( g_1^n \right)^{\tau_1} \cdots \left( g_h^n \right)^{\tau_h} \\
			&\leq \sum_{\abs{\tau} = K+1}^{\infty} \lambda^{\abs{\tau}} \left( g_1^n \right)^{\abs{\tau}} \\
			&\leq \sum_{k = K+1}^{\infty} \lambda^k h^k \left( g_1^n \right)^k \\
			&= \sum_{k = K+1}^{\infty} \left( \lambda h g_1^n \right)^k \\
			&= \left( \lambda h g_1^n \right)^{K+1} \cdot \frac{1}{1 - \lambda h g_1^n} \\
			&\leq 2 (\lambda h)^{K+1} \left( g_1^{K+1} \right)^n
		\end{align*}
		for $ n $ large enough.
		Therefore we choose $ K \geq 0 $ such that
		\begin{equation*}
			g_1^{K+1} < t.
		\end{equation*}
		So $ K $ is fixed now.
		We get
		\begin{equation*}
			2 (\lambda h)^{K+1} \left( g_1^{K+1} \right)^n < 2 (\lambda h)^{K+1} t^n = O(t^n).
		\end{equation*}
		Hence we have
		\begin{equation*}
			\alpha(n) = \sum_{\abs{\tau} = 0}^{K} A_{\tau} \left( g_1^n \right)^{\tau_1} \cdots \left( g_h^n \right)^{\tau_h} + O(t^n).
		\end{equation*}
		Defining $ \eta \in \overline{\QQ}\E_{\QQ} $ by
		\begin{equation*}
			\eta(n) = \sum_{\abs{\tau} = 0}^{K} A_{\tau} \left( g_1^n \right)^{\tau_1} \cdots \left( g_h^n \right)^{\tau_h}
		\end{equation*}
		satisfies
		\begin{equation*}
			\abs{\alpha(n) - \eta(n)} \ll t^n
		\end{equation*}
		for $ n $ large enough.
		By taking a new constant, the last bound holds for all $ n \in \NN $ and the lemma is proven.
	\end{proof}
	
	Now we have all tools together an can prove Theorem \ref{thm:severallrs}. We will proceed as in the proof of Theorem \ref{thm:approxresult}.
	
	\begin{proof}[Proof of Theorem \ref{thm:severallrs}]
		Put again $ t = 1/9 $.
		Lemma \ref{lem:approximp} gives a power sum $ \eta \in \overline{\QQ}\E_{\QQ} $ such that
		\begin{equation*}
			\abs{\alpha(n) - \eta(n)} \ll t^n.
		\end{equation*}
		The construction of $ \eta $ shows that we can write
		\begin{equation*}
			\eta(n) = w_1 d_1^n + \cdots + w_H d_H^n
		\end{equation*}
		for $ H \geq 1 $ (since $ \alpha \neq 0 $ by the assumption $ G_n^{(0)} \neq 0 $) as well as $ d_1,\ldots,d_H \in \QQ $ with $ 1 \geq d_1 > d_2 > \cdots > d_H > 0 $ and $ w_1,\ldots,w_H \in \overline{\QQ}^* $.
		Define $ k := H+1 \geq 2 $ and choose $ \varepsilon > 0 $ smaller than $ \frac{1}{2k} $.
		
		Assume indirectly that there are infinitely many values of $ n $ with corresponding $ (p,q) \in \ZZ^2 $ satisfying $ 0 < q < e^{n\varepsilon} $ such that
		\begin{equation}
			\label{eq:indirsecondthm}
			\abs{\alpha(n) - \frac{p}{q}} \leq \frac{1}{q^k} e^{-n\varepsilon}.
		\end{equation}
		
		Once again we aim to apply the subspace theorem.
		Therefore fix a finite set $ S $ of absolute values of $ \QQ $ containing the archimedean one and such that $ d_1,\ldots,d_H $ are $ S $-units.
		Moreover we define for each $ \nu \in S $ the $ H+1 $ linear forms $ L_{0,\nu}, \ldots, L_{H,\nu} $ in the $ H+1 $ variables $ X_0, \ldots, X_H $ as follows.
		For $ i=0,\ldots,H $ and $ \nu \in S $ with either $ i \neq 0 $ or $ \nu \neq \infty $ we put
		\begin{equation*}
			L_{i,\nu} := X_i
		\end{equation*}
		and for the remaining one
		\begin{equation*}
			L_{0,\infty} := X_0 - w_1 X_1 - \cdots - w_H X_H.
		\end{equation*}
		Obviously for each $ \nu \in S $ the linear forms $ L_{0,\nu}, \ldots, L_{H,\nu} $ are linearly independent and all linear forms have algebraic coefficients.
		Let $ d $ be the smallest positive rational integer such that $ dd_i \in \ZZ $ for all $ i=1,\ldots,H $.
		Then $ d $ is also an $ S $-unit.
		We set $ e_i := dd_i $ for $ i=1,\ldots,H $.
		So $ e_i \in \ZZ $ for $ i=1,\ldots,H $.
		Now we consider the vectors
		\begin{equation*}
			\uline{x} = \uline{x}(n,p,q) := (pd^n, qe_1^n, qe_2^n, \ldots, qe_H^n) \in \ZZ^{H+1}.
		\end{equation*}
		
		For the double product appearing in the subspace theorem  we get, in the same way as in the proof of Theorem \ref{thm:approxresult}, the upper bound
		\begin{equation}
			\label{eq:sstboundthree}
			\prod_{\nu \in S} \prod_{i=0}^{H} \abs{L_{i,\nu}(\uline{x})}_{\nu} \leq q^{H+1} \abs{\eta(n) - \frac{p}{q}}.
		\end{equation}
		
		Since $ q < e^{n\varepsilon} $, we have $ q^{-k} > e^{-kn\varepsilon} $ and consequently the bound $ q^{-k} e^{-n\varepsilon} > e^{-(k+1)n\varepsilon} $.
		Recalling $ \varepsilon < \frac{1}{2k} $ as well as $ k \geq 2 $ and $ t = 1/9 $ gives us
		\begin{align}
			q^{-k} e^{-n\varepsilon} &> e^{-(k+1)n\varepsilon} > e^{-n \frac{k+1}{2k}} \nonumber \\
			\label{eq:boundrttn}
			&= e^{-n \left( \frac{1}{2} + \frac{1}{2k} \right)} > e^{-n \left( \frac{1}{2} + \frac{1}{4} \right)} \\
			&= \left( e^{-3/4} \right)^n > \left( \frac{1}{3} \right)^n = \left( \sqrt{t} \right)^n. \nonumber
		\end{align}
		Now Lemma \ref{lem:approximp} gave us a constant $ L > 0 $ such that
		\begin{align*}
		\abs{\eta(n) - \frac{p}{q}} &\leq \abs{\eta(n) - \alpha(n)} + \abs{\alpha(n) - \frac{p}{q}} \\
		&\leq L t^n + \frac{1}{q^k} e^{-n\varepsilon} \\
		&= \left( L \left( \sqrt{t} \right)^n \right) \cdot \left( \sqrt{t} \right)^n + \frac{1}{q^k} e^{-n\varepsilon} \\
		&\leq \frac{2}{q^k} e^{-n\varepsilon}
		\end{align*}
		for $ n $ large enough and by using the bounds \eqref{eq:indirsecondthm} and \eqref{eq:boundrttn}.
		As $ 2 \leq e^{n\varepsilon / 2} $ for $ n $ large enough, we can update the bound \eqref{eq:sstboundthree} to
		\begin{align}
			\prod_{\nu \in S} \prod_{i=0}^{H} \abs{L_{i,\nu}(\uline{x})}_{\nu} &\leq q^{H+1} \abs{\eta(n) - \frac{p}{q}} \nonumber \\
			\label{eq:sstboundfour}
			&\leq 2 q^{H+1-k} e^{-n \varepsilon} \\
			&\leq e^{-n \varepsilon / 2}. \nonumber
		\end{align}
		
		Furthermore, by the same arguments as in the proof of Theorem \ref{thm:approxresult}, we get the inequality
		\begin{equation*}
			\abs{pd^n} \leq 2 \abs{w_1} q e_1^n
		\end{equation*}
		as well as the bound
		\begin{equation*}
			\max_{0 \leq i \leq H} \abs{x_i} \leq \left( e^{2\varepsilon} e_1 \right)^n
		\end{equation*}
		for $ n $ large enough.
		By choosing $ \delta > 0 $ small enough such that
		\begin{equation*}
			\delta < \frac{\varepsilon}{2 \log \left( e^{2\varepsilon} e_1 \right)},
		\end{equation*}
		we get
		\begin{equation*}
			\left( \max_{0 \leq i \leq H} \abs{x_i} \right)^{\delta} < e^{n \varepsilon / 2}.
		\end{equation*}
		Together with inequality \eqref{eq:sstboundfour} this yields
		\begin{equation*}
			\prod_{\nu \in S} \prod_{i=0}^{H} \abs{L_{i,\nu}(\uline{x})}_{\nu} \leq e^{-n \varepsilon / 2} < \left( \max_{0 \leq i \leq H} \abs{x_i} \right)^{-\delta}.
		\end{equation*}
		
		Since this inequality holds for infinitely many vectors $ \uline{x} = \uline{x}(n,p,q) $, Theorem \ref{thm:subspace} implies that infinitely many vectors $ \uline{x} $ are contained in a fixed proper subspace of $ \QQ^{H+1} $.
		Let
		\begin{equation*}
			z_0 X_0 - z_1 X_1 - \cdots - z_H X_H = 0
		\end{equation*}
		with $ z_0, z_1, \ldots, z_H \in \QQ $ be the defining equation of this subspace.
		Thus we have
		\begin{equation}
			\label{eq:secrelation}
			z_0 p d^n - z_1 q e_1^n - \cdots - z_H q e_H^n = 0
		\end{equation}
		for infinitely many triples $ (n,p,q) $.
		
		If $ z_0 = 0 $, then again we would have
		\begin{equation*}
			z_1 e_1^n + \cdots + z_H e_H^n = 0
		\end{equation*}
		for infinitely many $ n $.
		Since $ e_1, \ldots, e_H $ are distinct positive integers, this ends up in the contradiction $ z_1 = \cdots = z_H = 0 $.
		
		Therefore we can assume that $ z_0 = 1 $.
		Rewriting equation \eqref{eq:secrelation} gives us the equality
		\begin{equation}
			\label{eq:betan}
			\frac{p}{q} = z_1 d_1^n + \cdots + z_H d_H^n =: \beta(n).
		\end{equation}
		Note that $ \beta \in \QQ\E_{\QQ} $.
		We will show that in fact the characteristic roots of $ \beta $ with non-zero coefficient must be integers.
		Assume the contrary, i.e. there exists an index $ i $ such that $ d_i \in \QQ \setminus \ZZ $ and $ z_i \neq 0 $.
		Then the smallest positive integer $ D $ with the property that $ D^n \beta(n) \in \QQ\E_{\ZZ} $ satisfies $ D \geq 2 $.
		By Lemma \ref{lem:qversusz}, for all but finitely many $ n $ the denominator of $ \beta(n) $ is bounded from below by $ D^n e^{-n\varepsilon} $.
		Thus for infinitely many $ n $ we have
		\begin{equation*}
			q \geq D^n e^{-n\varepsilon} \geq 2^n e^{-n\varepsilon}.
		\end{equation*}
		Comparing this with the upper bound
		\begin{equation*}
			q < e^{n \varepsilon}
		\end{equation*}
		gives
		\begin{equation*}
			e^{\varepsilon} > q^{1/n} \geq 2 e^{-\varepsilon}
		\end{equation*}
		and
		\begin{equation*}
			2 < e^{2 \varepsilon} < e^{1/k} \leq e^{1/2}.
		\end{equation*}
		This is a contradiction, proving that $ d_i \in \ZZ $ if $ z_i \neq 0 $ for $ i=1,\ldots,H $.
		Hence we have $ \beta \in \QQ\E_{\ZZ} $.
		Recalling that $ 1 \geq d_1 > d_2 > \cdots > d_H > 0 $ are the characteristic roots of $ \beta $ implies
		\begin{equation*}
			\beta(n) = z_1 \in \QQ.
		\end{equation*}
		
		Using the equations \eqref{eq:indirsecondthm} and \eqref{eq:betan}, this gives us
		\begin{equation*}
			\abs{\alpha(n) - z_1} \leq \frac{1}{q^k} e^{-n\varepsilon} \leq e^{-n\varepsilon}.
		\end{equation*}
		Thus $ \alpha(n) $ converges to the rational number $ z_1 $.
		Since we have seen in the proof of Lemma \ref{lem:approximp} that $ \alpha(n) $ converges to a zero of $ F(0,\ldots,0,y) $, we get that $ F(0,\ldots,0,y) $ has a rational zero.
		This contradicts the assumption in the theorem, concluding the proof.
	\end{proof}

\end{document}